\documentclass[11pt]{amsart} 

\let\oldtocsection=\tocsection
\let\oldtocsubsection=\tocsubsection
\let\oldtocsubsubsection=\tocsubsubsection
 \renewcommand{\tocsection}[2]{\hspace{0em}\oldtocsection{#1}{#2}}
\renewcommand{\tocsubsection}[2]{\hspace{2em}\oldtocsubsection{#1}{#2}}
\renewcommand{\tocsubsubsection}[2]{\hspace{4em}\oldtocsubsubsection{#1}{#2}}
 
\usepackage{epigraph}
\usepackage[francais,english]{babel}
\usepackage{hyperref}
\usepackage{fontenc}
\usepackage{textcomp}
\usepackage[margin=1.5in]{geometry}
\usepackage{inputenc}
\usepackage[dvipsnames]{xcolor}
\usepackage{amsmath}
\usepackage{amsfonts}
\usepackage{amssymb}
\usepackage{pdfsync}
\usepackage{enumerate}

\makeatletter
\renewcommand{\@seccntformat}[1]{ {\csname
the#1\endcsname}.\hspace{0.2em}}
\makeatother

\def \ch {{\mbox{ch}}}
\def \rad {{\mbox{rad}}}
\def \Ad {{\mbox{Ad}}}

\def \C {{\mathbb C}}
\def\E{\mathbb E}
\def \R {{\mathbb R}}
\def \N {{\mathbb N}}

\def \1{{\mathds{1}}}

 \newtheorem{theo}{Theorem}[section]
\newtheorem{lemm}[theo]{Lemma} 
\newtheorem{prop}[theo]{Proposition}
\newtheorem{coro}[theo]{Corollary}
\newtheorem{defi}[theo]{Definition}

\newtheorem*{theo*}{Theorem}

\begin{document}
 \title[Brownian sheet and time inversion]{ Brownian sheet and time inversion\\From  $G$-orbit to $L(G)$-orbit}

\author{Manon Defosseux}
\maketitle 

 \begin{abstract}  We have proved in a previous paper that a space-time Brownian motion   conditioned to remain in a Weyl chamber associated to an affine Kac--Moody Lie algebra     is distributed as the radial part process of a Brownian sheet on the  compact real form of the underlying finite dimensional Lie algebra, the radial part being defined considering the coadjoint action of a loop group   on the dual of a centrally extended loop algebra. We present here a very brief proof of this result based on a time inversion argument and on    elementary  stochastic differential calculus.
  \end{abstract} 
 \section{Introduction} 
We propose here a short proof of the main result of \cite{defo3}. Let us briefly recall  this result. For this we need to consider a connected simply connected  simple compact Lie group $G$  and its Lie algebra $\mathfrak g$ equipped with an invariant scalar product  for the adjoint action of $G$ on $\mathfrak g$. One considers a standard Brownian sheet $\{x_{s,t}, s\in [0,1], t\ge 0\}$ with values in $\mathfrak g$ and for each $t>0$,  the process $\{Y_{s,t}, s\in [0,1]\}$ starting from the identity element of $G$ and  satisfying the stochastic differential equation (in $s$) 
$$t\, dY_{s,t}=Y_{s,t}\circ dx_{s,t},$$
where $\circ$ stands for the Stratonovitch integral. This is a  $G$-valued process. The adjoint orbits in $G$ are in correspondence with an alcove which is a fundamental domain for the action   on a Cartan subalgebra $\mathfrak t$ of $\mathfrak g$ of the extended  Weyl group associated the roots of G. We have proved in \cite{defo3} that if for any $t>0 $ one denotes by $\mathcal O(Y_{1,t})$ the element in the alcove corresponding to the orbit of $Y_{1,t}$ then   the random process\footnote{with the convention that $Y_{1,0}$ is  the identity element of $G$} 
$$\{(t,t\mathcal O(Y_{1,t})):t\ge 0\}$$ is a space-time brownian motion  in $\R\times \mathfrak t$ conditioned in Doob's sense to remain in a Weyl chamber which occurs in the framework of affine Kac--Moody    algebras \cite{Kac}. The proof of \cite{defo3} rests on a Kirillov--Frenkel character formula \cite{fre} from which follows an  intertwining  relation between the transition probability semi-group  of the Brownian sheet and the one of the conditioned process. Then a Rogers and Pitman's criteria \cite{pitmanrogers} can be applied, which provides the result.  The conditioned process obtained when $G=\mbox{SU}(2)$ plays a crucial role in \cite{boubou-defo} where a Pitman type theorem is proved for a real Brownian motion  in the unit interval. Time inversion is a key ingredient to get the Pitman type theorem in this case. In the present communication   a new proof of the main result of \cite{defo3} is proposed, which rests on such a time inversion firstly and secondly on an elementary but  nice property of the Brownian sheet  on $\mathfrak g$ and its wrapping on $G$.

The results presented    are valuable for themselves rather than for their proofs which are rudimentary. We present them  in section \ref{results} before giving the precise definitions of the objets that they involve.  The rest of the communication is organized as follows. In section \ref{Action orbits}  we recall the general framework of \cite{defo3}.  In particular we describe the coadjoint orbits of the loop group $L(G)$ in     the dual of the centrally extended loop algebra $L(\mathfrak g)$ and the Weyl chamber associated to such an infinite dimensional  Lie algebra which is an affine Kac--Moody algebra. In section \ref{LG and B} we define the radial process associated to the Brownian sheet on $\mathfrak g$ and recall the main theorem of \cite{defo3}. In  section \ref{O and B} we define two Doob conditioned processes   living respectively  in an alcove or in an affine  Weyl chamber, and  prove that the two  processes are equal up to a time inversion. Finally in section \ref{New proof} we propose a brief proof of the main result of \cite{defo3}.
\section{Statement of the results}\label{results}
Let us fix $\gamma$ in a fixed alcove associated to $G$ and consider $\{X^\gamma_{s,t}: s\in [0,1],t\ge 0\}$ a random sheet with values in  $G$, such that for any $t\ge 0$, 
\begin{align*}
\left\{
    \begin{array}{ll}
       X^\gamma_{s,t}=X_{s,t}\circ d(x_{s,t}+\gamma s)   \\
      X^\gamma_{0,t}=e,
    \end{array}
\right.
\end{align*}
where $e$ is the identity element of $G$. Then one has the three following statements, the second one being an immediate consequence of the first, and the last one being deduced from the second  by a time inversion argument.\\
 \\
 {\bf Satement 1 :} \\
The random process $\{X_{1,t}^\gamma : t\ge 0\}$ is a standard Brownian motion on $G$ starting from $\exp(\gamma)$.\\
\\
  {\bf Satement 2 :} \\
The random process $\{\mathcal{O}(X_{1,t}^\gamma ): t\ge 0\}$ is a standard Brownian motion  starting from $\gamma$ conditioned to remain in the alcove.\\
\\
  {\bf Satement 3 :}  \\
The radial part process $\{\rad(t\Lambda_0+\int_0^1(\cdot\vert d(x_{s,t}+\gamma st)): t\ge 0\}$
  is a space-time Brownian motion   with drift $\gamma$ conditioned to remain in an affine Weyl chamber.

 \section{Loop group and its orbits}\label{Action orbits}
 In this part we fix succinctely the general framework of the results.  One can find more details in \cite{defo3} and references therein for instance. Let $G$ be  a connected simply connected  simple compact Lie group and  $\mathfrak g$  its Lie algebra equipped with a Lie bracket denoted by $[\,\cdot\,,\cdot\,]_{\mathfrak{g}}$. We choose  a maximal torus  $T$ in $G$ and denote by  $\mathfrak t$ its Lie algebra. By compacity we suppose without loss of generality that   $G$ is a matrix Lie group. We denote by $\Ad$ the adjoint action of  $G$ on itself or on its Lie algebra  $\mathfrak g$ which is equipped with an $\Ad(G)$-invariant  scalar  product  $(\cdot\vert \cdot)$.  
 We consider the real vector space $L(\mathfrak g)$ of smooth loops defined on the unit circle $S^1$ with values in   $ \mathfrak g$,   $S^1$ being identified with $[0,1]$. We equip $L(\mathfrak g)$  with an $\Ad(G)$-invariant scalar product also denoted  by $(\cdot\vert \cdot)$ letting 
$$(\eta\vert \xi)=\int_0^1(\eta(s)\vert\xi(s))\, ds, \quad \eta,\xi\in L(\mathfrak g).$$
Equipped with the Lie bracket  $[\, \cdot\, ,\cdot\,]_{\mathfrak{g}}$  pointwise defined, $L(\mathfrak g)$ is a Lie algebra. We consider its central extension  
 $$\widetilde{L}(\mathfrak g)= {L}(\mathfrak g)\oplus \R c,$$  equipped with a Lie bracket  $[\,\cdot\,,\cdot\,]$ defined by 
 \begin{align}\label{Liehat} 
[\xi+\lambda c+ ,\eta+\mu c ]=[\xi,\eta]_{\mathfrak{g}}+ (\xi'\vert \eta)c,
\end{align}
for $\xi,\eta\in L(\mathfrak g)$, $\lambda,\mu\in \R$. 
We consider the   fundamental weight $\Lambda_0$ in   $\widetilde L(\mathfrak g)^*$ defined by 
$$\Lambda_0(L(\mathfrak g))=0, \quad \Lambda_0(c)=1,$$  the set of smooth loops $L(G)$ with values on $G$ and its  coadjoint action $\Ad^*$ on $ {L}(\mathfrak g)^*\oplus \R \Lambda_0$. It is  defined by \begin{align}\label{coadjointtilde} 
\Ad^*(\gamma)(\phi+\tau \Lambda_0)& =[\gamma.\phi-\tau (\gamma'\gamma^{-1}\vert \cdot)]+\tau\Lambda_0, \end{align} 
 for $ \gamma\in L(G), \phi\in L(\mathfrak g)^*, \tau\in \R,$
 where $(\gamma.\phi)(.)=\phi(\gamma^{-1}.\gamma)$.  We notice that the coordinate along $\Lambda_0$ of a linear form in $\widetilde L(\mathfrak g)^*$, which is called the level of the linear form, is not affected by the coadjoint action. 
\bigskip

\paragraph{\bf  Coadjoint $L(G)$-orbit.} For  $\zeta$ in $\widetilde L(\mathfrak g)^*$ we denoted by    $\widetilde{\mathcal O}_\zeta$ the coadjoint orbit $\Ad^*(L(G))\{\zeta\}$ in $\widetilde L(\mathfrak g)^*$. We have recalled in  \cite{defo3}  that roughly speaking for  $\xi$ in $ \widetilde{\mathcal{O}}_{\zeta}$,  provided that  $\zeta$   has a positive level, we find $\gamma\in L(G)$ such that $\xi=\Ad^*(\gamma)(\zeta)$ solving a differential equation.  Actually if   $\xi$ is written 
\begin{align}\label{formx}
\xi=\tau\Lambda_0+\int_0^1(\cdot\vert\dot{x}_s)\, ds,
\end{align}
with $\dot{x}\in L(\mathfrak g)$ and $\tau>0$, then   $\xi$   is in    $\widetilde{\mathcal O}_{\tau\Lambda_0+(a\vert \cdot)}$ for $a\in \mathfrak t$  if and only if
the $G$-valued function $\{X_s:s\in [0,1]\}$  starting from the identity element of $G$ and satisfying the differential equation
$$\tau d X=Xdx,$$   satisfies  $X_1\in \Ad(G)\{\exp(a/\tau)\}$ (see \cite{Segal} for  details or \cite{defo3} in which the results of \cite{Segal} are   recalled).  Thus     coadjoint orbits in the subspace of linear forms in $\tau\Lambda_0+L(\mathfrak g)^*$ written like in (\ref{formx}) are in one-to-one correspondence with the adjoint $G$-orbits in $G$. In order to parametrize these orbits, it is more convenient to consider the real roots of  $G$ rather than the infinitesimal ones. One can find for instance in chapters $5$ and $7$ of \cite{Brocker}   definitions and properties recalled in the next two paragraphs.
 
\bigskip
\paragraph{\bf Real roots.}    We consider the complexified Lie algebra $\C\otimes_\R \mathfrak g$ of $\mathfrak g$ that we denote by $\mathfrak g_\C$.  
The set of real roots is 
$$\Phi=\{\alpha\in \mathfrak t^*: \exists X\in \mathfrak{g}_\C\setminus \{0\}, \, \forall H\in \mathfrak t, \, [H,X]=2i\pi \alpha(H) X\}.$$  Suppose that $\mathfrak{g}$ is of rank $n$ and choose a set  of simple real roots $$\Pi=\{\alpha_k , \, k\in\{1,\dots,n\}\}.$$ 
We denote by 
$\Phi_+$ the set of positive real roots. The half sum of positive real roots is denoted by $\rho$. Letting for $\alpha\in \Pi$, $$\mathfrak{g}_\alpha=\{X\in \mathfrak{g}: \, \forall H\in \mathfrak{t},\, [H,X]=2i\pi\alpha(H)X\},$$ the coroot $\alpha^\vee$  of $\alpha\in \Phi$ is defined as the only vector of $\mathfrak{t}$ in $[\mathfrak{g}_\alpha,\mathfrak{g}_{-\alpha}]$ such that $\alpha(\alpha^\vee)=2$. For $\alpha\in \Pi$, one defines two transformations on $\mathfrak t$, the reflection  $s_{ \alpha^\vee}$ and the translation $t_{\alpha^\vee}$,  letting  for $x \in \mathfrak{t}$ 
$$s_{\alpha^\vee}(x)=x-\alpha(x)\alpha^\vee\, \textrm{ and } \, t_{\alpha^\vee}(x) =  x+ \alpha^\vee.$$  One  considers the Weyl group $W^\vee$ and the group $\Gamma^\vee$ respectively generated by reflections  $s_{\alpha^\vee}$ and  translations  $t_{\alpha^\vee},$ for $\alpha\in \Pi$, and the extended   Weyl group  $\Omega$ generated by $W^\vee$ and $\Gamma^\vee$. Actually $\Omega$ is the semi-direct product $W^\vee\ltimes  {\Gamma^\vee}$. A   fundamental domain for its action  on $\mathfrak{t}$ is  
\begin{align*}
A=\{x\in \mathfrak{t} : \forall \alpha\in \Phi_+, \, \, 0\le \alpha(x)\le  1\}.
\end{align*} 
$ $ 
 
\paragraph{\bf Adjoint $G$-orbit.}  The group  $G$ being  simply connected, the conjugaison classes  $G/\Ad(G)$ is in correspondence with the fundamental domain   $A$. Actually   for every    $\textsf u\in G$, there exists a unique element  $x\in A$ such that      $\textsf u\in \Ad(G)\{\exp(x)\}$.  For  $\tau\in \R_+$, one defines  the alcove $A_\tau$ of level  $\tau$ by 
$$A_\tau=\{x\in \mathfrak{t} : \forall \alpha\in \Phi_+, \,\, 0\le \alpha(x)\le \tau\},$$ i.e. $A_\tau=\tau A$.   In particular $A_1=A$.

 \bigskip
 
 \paragraph{\bf Alcoves and coadjoint $L(G)$-orbit.}  For a positive real number  $\tau$ and a linear form $\xi\in \widetilde{L}(\mathfrak g)^*$ written as in (\ref{formx})  there is a unique element in $a\in A_\tau$ such that  $$X_1\in \Ad(G)\{\exp(a/\tau)\}$$ where $X=\{X_s:s\in[ 0,1]\}$ starts from the identity element $e$ of  $G$ and satisfies 
$$\tau d X=Xdx. $$ Discussion above ensures that the pair $(\tau,a)$ determines the orbit of $\xi$. Thus     coadjoint orbits in the subspace of linear forms in $\R^*_+\Lambda_0+L(\mathfrak g)^*$ written like in (\ref{formx}) are in one-to-one correspondence with  $$\{(\tau, a)\in \R^*_+\times \mathfrak t: a\in A_\tau\}.$$ This last domain can be identified (if we add it $(0,0)$) with a Weyl Chamber associated to an affine Kac--Moody algebra as it is explained in the following paragraph.

\bigskip

 \paragraph{\bf Affine Weyl chamber} From now on the scalar product on $\mathfrak g$ is normalized such that $(\theta\vert\theta)=2$.   
We denote by   $\theta$ the highest real root and we let $\alpha_0^\vee=c-\theta^\vee$. We consider  $${\widehat{\mathfrak h}}=\mbox{Vect}_\C\{\alpha^\vee_0,\alpha_1^\vee,\dots,\alpha^\vee_n,d\} \, \textrm{ and } \,{ \widehat{\mathfrak h}}^*=\mbox{Vect}_\C\{\alpha_0,\alpha_1,\dots,\alpha_n,\Lambda_0\},$$  where $\alpha_0=\delta-\theta$ and for  $i\in\{0,\dots,n\}$
 $$\alpha_i(d)=\delta_{i0}, \quad \delta(\alpha_i^\vee)=0, \quad \Lambda_0(\alpha_i^\vee)=\delta_{i0},\quad \Lambda_0(d)=0.$$ 
 We let 
 $$\widehat{\Pi }=\{\alpha_i: i\in\{0,\dots,n\}\} \, \textrm{ and } \, \widehat{\Pi }^\vee=\{\alpha^\vee_i: i\in\{0,\dots,n\}\}.$$ Then $(\widehat{{\mathfrak h}},\widehat{\Pi},\widehat{\Pi}^\vee)$ is a    realization of a generalized Cartan matrix of   affine type. These objects are studied in details in   \cite{Kac}.    The following definitions mainly  come from chapters $1$ and $6$.   We consider the  restriction of $(\cdot\vert\cdot)$  to  $\mathfrak t$ and extend it to $\widehat{\mathfrak h}$ by $\C-$linearity   and   by letting 
$$(\C c+\C d\vert\,\mathfrak t)=0,\quad (c\vert c)=(d\vert d)=0, \quad (c\vert d)=1.$$  
 Then the linear isomorphism 
\begin{align*}
\nu:\, \,&\widehat{\mathfrak{h}}\to \widehat{\mathfrak{h}}^* \\
& h\mapsto (h\vert \cdot)
\end{align*} identifies $\widehat{\mathfrak{h}}$ and $\widehat{\mathfrak{h}}^*$. We still denote   $(\cdot\vert\cdot)$ the induced bilinear form     on $\widehat{\mathfrak{h}}^*$. We record that
\begin{align*}
&(\delta\vert \alpha_i)=0,\quad  i=0,\dots,n,\quad (\delta\vert\delta)=0,\quad (\delta\vert\Lambda_0)=1.
\end{align*} 
Due to the normalization we have $\nu(\theta^\vee)=\theta$ and $(\theta^\vee\vert\theta^\vee)=2$.   
We define the  affine Weyl group $\widehat{ W}$ as the subgroup  of  $\mbox{GL}({\widehat{\mathfrak{h}}}^*)$ generated by fondamental reflections   $s_{\alpha}$, $\alpha\in \widehat{\Pi}$, defined by $$s_{\alpha}(\beta)=\beta-  \beta(\alpha^\vee) \alpha,\ \quad \beta\in \widehat{{\mathfrak{h}}}^*.$$   
The bilinear form  $(\cdot\vert \cdot)$ is $\widehat{ W}$-invariant.
The affine Weyl group $\widehat{ W}$  is equal to the semi-direct product $ {W}\ltimes  {\Gamma}$, where $  W$ is the Weyl group of $G$ generated by $s_{\alpha_i}$, $i\in \{1,\dots,n\}$, and  $\Gamma$ the group of translations   $t_{\alpha}$, $\alpha\in \nu(  Q^\vee)$, defined by 
\begin{align} \label{transa}
t_\alpha(\lambda)=\lambda+\lambda(c)\alpha-\big[(\lambda\vert \alpha)+\frac{1}{2}(\alpha\vert\alpha)\lambda(c)\big]\delta, \quad \lambda\in \widehat{{\mathfrak{h}}}^*.
\end{align}
Identification of $\widehat{\mathfrak h}$ and $\widehat{\mathfrak h}^*$  via $\nu$ allows to define an action of $\widehat W$ on $\widehat{\mathfrak h}$. One lets 
  $wx=\nu^{-1}w\nu x$, for $w\in \widehat W$, $x\in \widehat{\mathfrak h}$. Then the action of $\widehat W$ on $\Lambda_0\oplus \mathfrak t^*\oplus \R\delta/\R \delta$ or $d\oplus \mathfrak t\oplus \R c/\R c$ is identified to the one of $\Omega$ on $\mathfrak t$.  Moreover a   fondamental domain for the action of $\widehat{ W}$ on the quotient space $(\R_+\Lambda_0+ \mathfrak t^*+ \R\delta)/\R\delta$ is 
  $$\{\lambda\in \R\Lambda_0\oplus \mathfrak t^*: \lambda(\alpha^\vee)\ge 0, \alpha\in\widehat{ \Pi}\},$$ and for $\tau\ge 0$,    $\tau\Lambda_0+\phi_a$, with $\phi_a=(a\vert \cdot)$, is in    this fundamental domain  if and only if   $a\in A_\tau$.  Then we  consider the following domain which is identified with the  fundamental affine Weyl chamber viewed in the quotient space
  $$C_W=\{(\tau,x)\in \R_+\times \mathfrak t: x\in A_\tau\}.$$

\section{Coadjoint $L(G)$-orbit and Brownian motion}\label{LG and B}
When $\{x_s: s\in[ 0,1]\} $ is a continuous semi-martingale    with values in  $\mathfrak g$, then for $\tau>0$ the stochastic differential equation
 \begin{align}\label{eds}
\tau \, dX=X\circ dx,
\end{align} 
where $\circ$ stands for the Stratonovitch integral, has a unique solution starting from $e$.  Such a solution is a $G$-valued process, that we   denote by $\epsilon(\tau,x)$ \cite{karandikar},\cite{lepingle}. This is the Stratonovitch stochastic exponential of $\frac{x}{\tau}$. The previous discussion leads naturally to the  following definition.
 \begin{defi}  
For $\tau\in \R_+^*$, and $x=\{x_s: s\in[0,1]\}$ a $\mathfrak g$-valued continuous semi-martingale, we   defines the radial part of $\tau\Lambda_0+\int_0^1(\cdot\vert dx_s)$  that we denote by  $\rad(\tau\Lambda_0+\int_0^1(.\vert dx_s))$  by\footnote{We do not specify in which space lives this distribution. We use this notation here just to keep track of the fact that when $x$ is a Brownian motion the Wiener measure   provides a natural measure on a coadjoint orbit in the original work of I. B. Frenkel.} 
$$\rad(\tau\Lambda_0+\int_0^1(\cdot\vert dx_s))=(\tau,a),$$  
 where $a$ is the  unique element in $ A_\tau$  such that $$\epsilon(\tau,x)_1\in \Ad(G) \{\exp(a/\tau)\}.$$  \end{defi}
 We have proved in \cite{defo3} the following theorem, where the conditioned space-time Brownian motion is the one defined in section \ref{CBMWC}. This is this theorem for which we propose a new proof.
  \begin{theo}\label{maintheo} 
    If $\{x_{s,t}: s\in [0,1],t\ge 0\}$ is a Brownian sheet with values in $\mathfrak g$ such that  for any $a,b\in \mathfrak g$, $s_1,s_2\in [0,1]$, $t_1,t_2\in \R_+^*$, $$\E\big(( a\vert x_{s_1,t_1})( b\vert x_{s_2,t_2} )\big)=\min(s_1,s_2)\min(t_1,t_2)(a\vert b),$$ 
 then 
  $$\{\rad(t\Lambda_0+\int_0^1(\cdot\vert dx_{s,t})): t\ge 0\}$$
  is a space-time Brownian motion in $\R\times \mathfrak t$ conditioned to remain in the affine Weyl chamber $C_W$.
  \end{theo}
\section{Conditioned brownian motions} \label{O and B}
In whole the communication, when we write $f_t(x) \propto g_t(x)$ for $f_t(x),g_t(x)\in \C$, we mean that $f_t(x)$ and $g_t(x)$ are equal up to a multiplicative constant independent of the parameters $t$ and $x$. 
\subsection{A Brownian motion conditioned to remain in an alcove}\label{CBMA}
There is a common way to construct a Brownian motion conditioned in Doob sense to remain in an alcove, which is to consider at each time the $\Ad(G)$-orbit of a brownian motion in $G$. The brownian motion on $G$ is left Levy process. Its transition probability densities  $(\bold{p}_s)_{s\ge 0}$   with respect to the Haar measure on $G$    can be  expanded  as a sum of characters of   highest-weight complex representations of $G$. These representations  are in correspondence with 
$$P_+=\{\lambda\in \mathfrak t^*: \lambda(\alpha_i^\vee)\in \N, i\in \{0,\dots,n\}\}.$$ One has  for $s\ge 0$, $\textsf u,\textsf v\in G$,
\begin{align*}
\bold{p}_s(\textsf{u},\textsf v)=\bold{p}_s(e,\textsf u^{-1}\textsf v)=\sum_{\lambda\in P_{+}}\ch_\lambda(e)\ch_\lambda(\textsf u^{-1}\textsf v)e^{-\frac{s  (2\pi )^2}{2} (\vert\vert\lambda+\rho\vert\vert^2 -\vert\vert \rho\vert\vert^2)}, 
\end{align*}
where $\ch_\lambda$ is the character of the irreducible representation of  highest weight  $\lambda$  (see for instance \cite{Fegan}). By the Weyl character formula one has\footnote{The presence of a factor $2i\pi$ is due to the fact that we have considered the real roots rather than the infinitesimal ones.} for $h\in\mathfrak t$
\begin{align}\label{Weyl}
\ch_\lambda(e^{h})=\frac{\sum_{w\in W}\det(w)e^{2i\pi\langle w(\lambda+\rho),h\rangle}}{\sum_{w\in W}\det(w)e^{2i\pi\langle w( \rho),h\rangle}}.
\end{align}
 We let $$\pi(h)=\prod_{\alpha\in \Phi_+}\sin\pi  \alpha(h) , $$
which is the denominator in (\ref{Weyl}). 
Such a    process starting from $\textsf u\in G$ can be obtained considering a standard Brownian motion $\{x_s: s\ge 0\}$ with values in $\mathfrak g$, and  the solution $\{X_s:s\ge 0\}$ of the stochastic differential equation 
$$dX=X\circ dx$$ with initial condition $X_0=\textsf u$. Then $\{X_s: s\ge 0\}$ is a standard Brownian motion on $G$ starting from $\textsf u$.  If  $\textsf u=\exp(\gamma)$ with $\gamma\in A$ then   the process $\{r^\gamma_s: s\ge 0\}$ such that for any $s\ge 0$, $r^\gamma_s$ is the unique element in $A$ such that   
$$X_s\in \Ad(G)\{\exp(r^\gamma_s)\},$$ is a Markov process starting from $\gamma$ with transition probability densities  $(q_t)_{t\ge 0}$ with respect to the Haar measure on $G$ given by
\begin{align}\label{qA1}
q_t(x,y)\propto \pi(y)^2\sum_{\lambda\in P_+} \ch_\lambda(e^{-x})\ch_\lambda(e^{y})  e^{-\frac{t  (2\pi)^2}{2}(\vert\vert\rho+\lambda\vert\vert^2-\vert\vert \rho\vert\vert^2)}, 
\end{align}
for $ t\ge 0,$ $x,y\in A$. This is obtained integrating over an $\Ad(G)$-orbit (see  $(4.3.3)$ in \cite{fre} for instance) and using the Weyl integration formula.
This Markov process is actually a Brownian motion killed on the boundary of $A$ conditioned never to die. In fact if we denote by  $(u_t)_{t\ge 0}$    the transition  densities  of the standard Brownian motion  on $\mathfrak t$ killed on the boundary of $A$, a reflection principle  gives that for $t>0 $, $x,y\in A$, 
\begin{align}\label{reflA}
u_t(x,y)=\sum_{w\in \Omega}\det(w)p_t(x,w(y)),
\end{align}
where $p_t$ is the  standard heat kernel on $\mathfrak t$ and $\det(w)$ is   the determinant of the linear part of $w$. A Poisson summation formula (see \cite{Bellman} for general results, and \cite{fre} or \cite{defo3} for this particular case) then shows that 
\begin{align}\label{qA2}
q_t(x,y)\propto\frac{\pi(y)}{\pi(x)}e^{2\pi^2(\rho\vert\rho)t}u_t(x,y),
\end{align}
which is the transition probability of the killed Brownian motion conditioned in the sense of Doob to remain in $A$.  

\subsection{A space-time Brownian motion conditioned to remain in an affine Weyl chamber}\label{CBMWC}
We define a space-time Brownian motion conditioned to remain in an affine Weyl chamber as it has been defined in \cite{defo3} and also in \cite{boubou-defo} when $G=\mbox{SU}(2)$. It is defined as an $h$-process, with the help of an anti-invariant classical theta function.   
For $ \tau\in \R_+^*,$ $b\in \mathfrak t$, $a\in A_\tau$, we define $\widehat{\psi}_{b} (\tau,a)$ by 
\begin{align}\label{psi}
\widehat{\psi}_{b} (\tau,a)=\frac{1}{\pi(b)}\sum_{w\in \widehat{ W}}\det(w)e^{\langle w(\tau\Lambda_0+\phi_a),  d+b\rangle}.
\end{align}
From now on we fix $\gamma\in A$. One  considers  a standard Brownian motion $\{b_t:t\ge 0\}$ with values in $\mathfrak t$,   the space-time Brownian motion $\{B^\gamma_t=(t,b_t+\gamma t): t\ge 0 \}$, and the stopping time $T=\inf\{t\ge 0:   B^\gamma_t\notin   C_W\}$.     One defines a function $\Psi_\gamma $ on $   C_W $ by  
\begin{align}\label{Psi1}
\Psi_\gamma :(t,x)\in   C_W\to e^{-( \gamma\vert x)}\psi_{ {\gamma}}(t,x).
\end{align}
 Identity (\ref{reflA}) and decomposition $\widehat W=W\ltimes \Gamma$  implies that  
\begin{align}\label{Psi2}
\Psi_\gamma(t,x)\pi(\gamma)\propto t^{-n/2} u_{\frac{1}{t}}(\gamma,x/t)e^{\frac{t}{2}\vert\vert\gamma-x/t\vert \vert^2}
\end{align}

\begin{prop} The function $\Psi_\gamma $  is  a  constant sign  harmonic function for the killed process $\{  B^\gamma_{t\wedge T}: t\ge 0\}$, vanishing   on the boundary of  $  C_W$. 
\end{prop}
\begin{proof} The fact that $\Psi_\gamma $    is harmonic and satisfies the boundary conditions  is clear from (\ref{Psi1}). It is non negative by  (\ref{Psi2}). 
\end{proof}
 
\begin{defi}
We define $\{A^\gamma_t=(t,a^\gamma_t): t\ge 0\}$ as the killed process $\{  B^\gamma_{t\wedge T}: t\ge 0\}$ starting from $(0,0)$ conditionned in Doob's sense not to die, via the harmonic function 
$\Psi_{ {\gamma}}$. \end{defi}
More explicitely, if we  let for $t\ge 0$, $K^\gamma_t=B^\gamma_{t\wedge T}$, and $K^\gamma_t=(t,k^\gamma_t)$, then $\{A^\gamma_t=(t,a^\gamma_t) : t\ge 0\}$  is a Markov process starting from $(0,0)$ such that for $r,t> 0$,  the probability density of $a^\gamma_{t+r}$ given that $a^\gamma_r=x$, with $x\in A_r$, is 
\begin{align}
s_t^\gamma((r,x),(r+t,y))=\frac{\Psi_{ {\gamma}}(r+t,y)}{\Psi_{ {\gamma}}(r,x)}w_t^\gamma((r,x),(r+t,y)), \, \, (r+t,y)\in   C_W,
\end{align}
where $w_t^\gamma((r,x),(r+t,\cdot))$ is the probability density of $k^\gamma_{r+t}$ given that $k^\gamma_t=x$,   and the probability
  density of $a_t^\gamma$ is given by 
\begin{align} \label{entrancedrift}
s^\gamma_t((0,0),(t,y))=C_t\Psi_\gamma(t,y)\pi(\frac{y}{t})e^{-\frac{1}{2t}\vert \vert y-\gamma t\vert\vert^2}, \quad  y\in   A_t,\end{align}
where $C_t$ is a normalizing constant depending on $t$. 
\subsection{The two conditioned processes and time inversion} 
Actually the two Doob transformations previously defined are equal up to a time inversion. We prove this property as it is done in \cite{boubou-defo} for the Brownian motion in the unit interval. The following lemma is immediately deduced from (\ref{Psi2}) and (\ref{entrancedrift}). 
\begin{lemm}
For $t>0$, $x\in A$, one has 
$$s_{1/t}^\gamma((0,0),(1/t,x/t))=q_t(\gamma,x).$$
\end{lemm}
\begin{lemm} For $0< r\le t$, $x\in A_r,$ $y\in A_t$
$$e^{-\frac{1}{2t}\vert\vert y\vert \vert^2}u_{\frac{1}{r}-\frac{1}{t}}(y/t,x/r)=e^{-\frac{1}{2r}\vert\vert x\vert \vert^2}w_{t-r}^0((r,x),(t,y)).$$
\end{lemm}
\begin{proof} Using expression (\ref{reflA}) and the time inversion invariance property for the standard heat kernel on $\mathfrak t$, one obtains that 
$$e^{-\frac{1}{2t}\vert\vert y\vert \vert^2}u_{\frac{1}{r}-\frac{1}{t}}(y/t,x/r)=e^{-\frac{1}{2r}\vert\vert x\vert \vert^2}\sum_{w\in \Omega}e^{-\frac{1}{2t}(\vert\vert y\vert \vert^2-\vert\vert tw(y/t)\vert \vert^2)}p_{t-r}(x,tw(y/t)).$$
The sum on the right-hand size of the identity is exactly $w_{t-r}^0((r,x),(t,y))$ according to lemma $6.3$ of  \cite{defo2}, which achieves the proof.
\end{proof}
In the following proposition  $\{r^\gamma_t: t\ge 0\}$ is the conditioned process   defined in section \ref{CBMA} and $\{a^\gamma_t :t\ge 0\}$ is the one defined in section \ref{CBMWC}. 
\begin{prop}\label{inv}  One has in distribution 
$$\{ta^\gamma_{1/t}: t\ge 0\} \overset{d}=\{r^\gamma_t: t\ge 0\}.$$
\end{prop}
\begin{proof} It follows immediately from the two previous lemmas and identity (\ref{Psi2}).
\end{proof}

\section{A new proof of Theorem \ref{maintheo}} \label{New proof}
For every $t>0$ one considers the diffusion process $\{Y^\gamma_{s,t}:s\in[0,1]\}$ starting from the identity element $e$  of $G$ satisfying  the EDS (in $s$)
\begin{align*}t d Y^\gamma_{s,t}=Y^\gamma_{s,t} \circ d(x_{s,t} +\gamma s t).
\end{align*}
For $\textsf u\in G$ one denotes by $\mathcal O(\textsf u)$ the unique element in $A$ such that $$\textsf u\in \Ad(G)\{\exp(\mathcal O(\textsf u))\}.$$ 
We have proved in \cite{defo3} that the random process $\{(t,t\mathcal O(Y^0_{1,t})): t\ge 0\}$ is distributed as $\{A^0_t: t\ge 0\}$. As $Y^\gamma$ satisfies
$$ d Y^\gamma_{s,t}=Y^\gamma_{s,t} \circ d(\frac{1}{t}x_{s,t} +\gamma s ),$$
and $\{\frac{1}{t}x_{s,t} : s,t>0 \}\overset{d}=\{x_{s,1/t}: s,t>0 \}$, one could deduce from \cite{defo3}, with the help of a Kirillov-Frenkel  character formula from \cite{fre} and a Cameron--Martin theorem, that the result remains true for any $\gamma\in A$.   

We propose here a   brief proof of  the theorem, which is valid for every $\gamma$. For every $t\ge 0$, one considers the diffusion process $\{X^\gamma_{s,t} :s\in[ 0,1]\}$ starting from $e\in G$    satisfying the stochastic differential equation (in $s$)
\begin{align}\label{strato}
dX_{s,t}^{\gamma}=X_{s,t}^{\gamma}\circ d(x_{s,t}+\gamma s).\end{align} 
\begin{prop}\label{mainprop}  $ $
\begin{enumerate} 
\item For $t,t'\ge 0$, the random process $\{X_{s,t+t'}^{\gamma}(X_{s,t}^{\gamma})^{-1}:s\in[ 0,1]\}$ has the same law as $\{X_{s,t'}^{0}: s\in[ 0,1]\}$.
\item For $t,t'\ge 0$, the random process  $\{X_{s,t+t'}^{\gamma}(X_{s,t}^{\gamma})^{-1}:s\in[ 0,1]\}$  is independent of  $\{X_{s,r}^{\gamma}:s\in[ 0,1], r\le t\}$.
\item  The random process $\{X^\gamma_{1,t}:t\ge 0\}$ is a standard Brownian motion in $G$ starting from $\exp(\gamma)$.
\end{enumerate}
\end{prop}
\begin{proof} For the first point, we  let $Z_s=X_{s,t+t'}^{\gamma}(X_{s,t}^{\gamma})^{-1}$, $s\in [0,1]$.   The process $\{(X_{s,t}^\gamma)^{-1} :s\in[ 0,1]\}$   satisfies  the EDS (in $s$) 
$$d(X_{s,t}^{\gamma})^{-1} = -d(x_{s,t}+\gamma s)\circ (X_{s,t}^\gamma)^{-1} $$
 from which we immediately deduce that $Z$ satisfies $$dZ_s=Z_s \circ X_{s,t}^\gamma\, d(x_{s,t+t'}-x_{s,t})(X_{s,t}^\gamma)^{-1}.$$ As $\{\int_0^sX_{r,t}^\gamma\, d(x_{r,t+t'}-x_{r,t})(X_{r,t}^\gamma)^{-1}: s\in[ 0,1]\}$ has the same law as $\{x_{s,t'}: s\in[ 0,1]\}$, and is independent of $\{x_{s,r}: s\in[ 0,1], r\le t\}$, one gets the first two points,
which imply in particular that $\{X_{1,t}^\gamma: t\ge 0\}$ is a right  Levy process. The $\Ad(G)$-invariance of the increments law implies that it is also a left Levy process. As for any $t>0$, $X_{1,t}^0$ and $X_{t,1}^0$ are equal in distribution,  the third point follows.

\end{proof}
Proposition \ref{mainprop} has the two following corollaries, the second one being deduced from the first by proposition \ref{inv}.
\begin{coro} The random process $\{\mathcal O(X^\gamma_{1,t}): t\ge 0\}$ is a standard Brownian motion starting from $\gamma$  killed on the boundary of $A$   conditioned in Doob's sense to remain in  $A$.
\end{coro}
\begin{coro} \label{maincor} The random process $\{(t,t\mathcal O(X_{1,1/t}^\gamma)): t\ge 0\}$ has the same distribution as the conditioned process $\{A^\gamma_t: t\ge 0\}$.
\end{coro}
As the two processes 
$\{x_{s,1/t}: s,t>0\}$ and $ \{\frac{1}{t}x_{s,t}: s,t>0\}$ are equal in distribution, Theorem \ref{maintheo} follows from 
 corollary \ref{maincor} with $\gamma=0$. For any $\gamma\in A$, one has under the same hypothesis as in the theorem  the following one.
  \begin{theo}\label{maintheogamma} 
 The radial part process $$\{\rad(t\Lambda_0+\int_0^1(\cdot\vert d(x_{s,t}+\gamma st))): t\ge 0\}$$
  is distributed as the Doob conditioned process $\{A^\gamma_t: t\ge 0\}$. \end{theo}

\end{document}